\RequirePackage[hyphens]{url}
\documentclass{article}
\usepackage{breakurl}
\PassOptionsToPackage{hyphens}{url}
\usepackage{amsmath, amsthm, amscd, amsfonts, amssymb, graphicx, color}
\usepackage{enumerate}

\theoremstyle{theorem}
\newtheorem{theorem}{Theorem}
\newtheorem{corollary}{Corollary}
\theoremstyle{corollary}

\newtheorem{lemma}{Lemma}
\newtheorem{definition}{Definition}
\newtheorem{remark}{Remark}
\theoremstyle{Definition}

\newtheorem{example}{Example}
\theoremstyle{Example}
\newtheorem{proposition}{Proposition}

\theoremstyle{Remark}
\numberwithin{equation}{section}

\newcommand{\tref}[1]{Theorem~\textup{\ref{#1}}}

\newcommand{\cref}[1]{Corollary~\textup{\ref{#1}}}
\newcommand{\lref}[1]{Lemma~\textup{\ref{#1}}}
\newcommand{\rref}[1]{Remark~\textup{\ref{#1}}}
\newcommand{\dref}[1]{Definition~\textup{\ref{#1}}}

\author{Antoine Mhanna}
\title{On Families of Numerical Semigroups with two Coprime Generators and Dimension three}
\date{Kfardebian, Lebanon, tmhanat@yahoo.com}
\begin{document}
\maketitle

\begin{abstract}
This paper examines in a new way  some  known facts about numerical   semigroups  especially  when the number of minimal generators  (that is the embedding dimension) is at most three and  at least two minimal generators are coprime. For such semigroups, an algorithm  is re-investigated to find  the pseudo-Frobenius numbers. A certain family  of $n$ dimensional  numerical semigroups of type at most $n-1$ is also given.  \end{abstract}
\section{Introduction and preliminaries}
 	Let $a_1,\ldots,a_n$ be  $n$ positive integers with $\gcd(a_1,\ldots,a_n)=1$. The set 
$\displaystyle S=\Bigg\{\sum_{i=1}^n{\lambda}_ia_i\big{|} n\in{\mathbb{N}}, {\lambda}_i \ge 0, \text{for all } i\Bigg\}$ is called the numerical  semigroup $S$ and $a_1,\ldots,a_n$ are called the generators of $S$. They are  minimal generators if we cannot take out a generator $a_i$ without changing the set $S$; in general we denote $S$ by   $\langle a_1,\ldots,a_n\rangle$.  For a numerical semigroup $S\neq \mathbb{N}$ the number $F(S):=\max\{n\in {\mathbb{N}}| n\notin S\}$ exists (see \cite[Theorem 1.0.1]{fon} for   proof) and is called  the  Frobenius number of $S$. We will suppose hereafter that $a_1<a_2$. 
The interested reader can see \cite{w}   and the references therein (in particular \cite{rou}) for an  extensive reading on numerical semigroups. 

There are different algorithms   computing   $F(S)$ with $S$ a three minimally generated semigroup; among them is  \cite{joh}, also \cite{san,fel}, and R\"odseth's  formula \cite{z}.  We refer to  \cite{fon} for  a further discussion on these and other  generalized  algorithms.

In this section we revisit the classical semigroup $\langle a_1,a_2\rangle$. We will represent a proof from \cite[p.\ 134]{a}  for two main purposes: one,  this makes the article more self-contained and two, reformulates in details an often cited result of Sylvester \cite{a} on the Frobenius problem.

\begin{definition}\label{lll}Let $a_1$ and $a_2$ be positive integers with $\gcd(a_1,a_2)=1$.
The  set of positive integers $x$ of the form $x=\alpha a_i-\beta a_{3-i}$, $i=1$ or $2$ with $0<\alpha <a_{3-i}$ and $0<\beta <a_i$  is called  the non-compound set and is denoted by $NC$.
The  set of integers $x$ of the form $x=\alpha a_1+\beta a_2$ with $\alpha >0,$ $\beta >0$, and $x<a_1a_2$  is called the compound set  and is denoted  by $C$.
\end{definition}
\begin{lemma}\cite{bran}\label{Se} Let  $S=\langle a_1,a_2\rangle$  and let $x$ be a positive integer;  we have $x\in NC$ if and only if  $0<x<a_1a_2$ and $x \notin S$.

\end{lemma}
\begin{proof} Suppose  $x\in S$ and in $NC$ with $x<a_1a_2$. Consequently,  $x=aa_1+ba_2=\alpha a_i-\beta a_{3-i}$ for some non-negative numbers $a,b$;    but this contradicts the fact that  $\alpha <a_{3-i}$ and $\beta <a_i$ (for example if $i=1$, we have $ (a-\alpha)a_1=-(\beta+b)a_2$ and $|a-\alpha|<a_2)$.
Take the positive integers less than $a_1a_2$ having neither $a_1$ nor $a_2$ as a divisor, there are $(a_1-1)(a_2-1)$  such integers. Arrange these numbers as pairs $(x,y)$ summing $a_1a_2$;
if $x\in NC$, then $a_1a_2-x=y \in S$ and 
 if $y \in C$, then $(a_1a_2-y)=x\notin S$, which means that  $   x \in NC$.
\end{proof}

We have proven that $|NC|=|C|=\dfrac{(a_1-1)(a_2-1)}{2}. $  The number $\mathcal{F}(S)=(a_1-1)(a_2-1)-1=a_1a_2-a_1-a_2$ does not belong to $S$  since $a_1a_2-\mathcal{F}(S) $ is  compound. 
 Any number $y>a_1a_2$ can be written as $\alpha a_1a_2+r$ and it is easily seen that since $r<a_1a_2$, the remainder $r$ is either in $S$ or in $NC$. Therefore  all $n>\mathcal{F}(S)$ do belong to $S$. 

Now  we can characterize all the numbers that are not in $S=\langle a_1,a_2\rangle$.
  \begin{proposition}\label{po}The set $NC$ is the set $\{\mathcal{F}(S)-y| y\in S, y<\mathcal{F}(S)\}$.
\end{proposition}
\begin{proof}A direct application of previous result.
\end{proof}
\begin{remark}\label{rk}
Notice at last that any integer $x < a_1a_2$ in $S$  has a unique representation as $x=\alpha a_1 + \beta a_2$ with $(\alpha, \beta) \in \mathbb{N}^2$, $\alpha< a_2$, and $\beta< a_1$; assuming the contrary  $\alpha a_1 + \beta a_2=\nu a_1 + \gamma a_2$, that is $(\alpha -\nu)a_1 =(\gamma-\beta) a_2$ and we can not  have $a_2$ divides $(\alpha-\nu)$ or $a_1$ divides $(\gamma-\beta)$.
A number $x$ in $NC$  can be written as $a_1 a_2-wa_1-ra_2=(a_2-w) a_1-r a_2$, where $1\le w<a_2 $ and $1\le r <a_1$. Hence, the uniqueness of $(w,r)\in \mathbb{N}^2$ follows similarly by contradiction.\end{remark}

We end this expository section by the following definition.
\begin{definition}
For a numerical semigroup $S$, we have $T(S):=\{x\in \mathbb{N}| x\not\in S, x+s\in S, \text{ for all } s\in S, s>0\}$. The cardinality of $T(S)$, denoted $t(S)$, is called the type of $S$  and a number in $T(S)$ is called a pseudo-Frobenius number.

The Ap\'ery set of $S$ with respect to $n\in S$ is the set ${\rm{Ap}}(S,n)=\{s\in S|s-n\notin S\}$ and the genus of 
$S$, denoted $g(S)$, is the cardinality of  $\{\mathbb{N}\backslash S\}$.

\end{definition}
  
\section{Adding  minimal generator}
Here we consider numerical semigroups $S$ minimally generated by $a_1,a_2,a_3$; the integers $a_1$ and $a_2$ are coprime with $a_1<a_2$. The next remark is already well known and  is very useful, see \cite[Chapter 1]{w}.
\begin{remark}\label{w}\cite{joh,rou}
 Let $a_1,a_2,\ldots,a_n$ be positive integers. If $\gcd(a_2,\ldots,a_n ) = d$ and
$a_j=d a'_j$ for each $j>1$, then the following holds:
\begin{itemize}
\renewcommand\labelitemi{$\scriptstyle\bullet$}
\item $F(\langle a_1,a_2,\ldots,a_n\rangle)=d F(\langle a_1,a'_2,\ldots,a'_n\rangle)+a_1(d-1)$  (Johnson's formula). 
\item The type of $\langle a_1,a_2,\ldots,a_n\rangle$ equals type of $\langle a_1,a'_2,\ldots,a'_n\rangle$.
\item The type of $S=\langle a_1,a_2,a_3\rangle$ is at most two.
\item We have the equivalence: $$w\in{\rm{Ap}}(\langle a_1,a'_2,\ldots,a'_n\rangle,a_1)\Longleftrightarrow  dw\in {\rm{Ap}}(\langle a_1,a_2,\ldots,a_n\rangle,a_1).$$
\end{itemize}
\end{remark}
\begin{remark}\label{qq} A number $x$ in $NC$ is written as $a_1 a_2-k a_2-j a_1$ and we have  $0<k<a_1$ if and only if  $0<j<a_2$. This is useful when $ix \in NC$  $(i\in \mathbb{N}^*)$ and $0<i k<a_1$ so $0<ij-(i-1)a_2<a_2$ or $0<i j<a_2$ so $0<i k-(i-1) a_1<a_1$.
\end{remark}
From the unique   representation of $NC$  numbers we have \lref{ert}.
\begin{lemma} \label{ert}Let $S=\langle a_1,a_2,a_3\rangle$, $Q=\langle a_1,a_2\rangle$, and $a_3=a_1a_2-k a_2-ja_1$ for some $(k,j)$ with $1\le k< a_1$ and $1\le j < a_2$. If  $2 a_3 \in  Q$, then the possible pseudo-Frobenius numbers are $\{a_1 a_2-(k+1) a_2-a_1,a_1 a_2-a_2-(j+1) a_1\}$. 
\end{lemma}

\begin{corollary}Let $S=\langle a, ka+d, ha+2 d\rangle$ with $d\ge 1$ and $h$ a non-negative integer. For $a=2i+1$ odd, if $\gcd(a,d)=1$ and  $\lfloor\frac{h}{2}\rfloor<k<h(\frac{a+1}{2})+d$,   then  \\$T(S)= \{F_1:=a(ha+2d)- a- (\frac{a-1}{2}+1)(ha+2d) , F_0:= a(ha+2d)-(h \frac{a-1}{2}+d+1-k+h)a-  (ha+2d)\}$. \end{corollary}
\begin{proof}Write first $$ka+d=a(ha+2d)- \Big(h \frac{a-1}{2}+d-k+h\Big)a-\Big(\frac{a-1}{2}\Big)(ha+2d)$$ and $$ha+2d=2(ka+d)-(2k-h)a. $$ From these expressions and the conditions stated, we know that the generators of $S$ are minimal. Take $Q=\langle a,h a+2 d\rangle$, where $a_1=a$  and $a_2=(ha+2 d)$. Apply \lref{ert}; since $F_0+a_3=(2k-h-1)(a_1)+i(a_2)\in Q$ and $F_1+a_3=a_1(d-1+hi+k)\in Q$  we get the result. 
\end{proof}

The case $k=h$ is well known \cite{tm}. See \cite[Theorem 5.4.3]{fon}  and  \cite{tm} for  a general family of  numerical semigroups  with an   arithmetic progression of generators.

\begin{definition}\label{km}Let $S=\langle a_1,a_2,a_3\rangle$ and $Q=\langle a_1,a_2\rangle$.
If  $x:= a_1a_2-k a_2-j a_1$,  $1\le k<a_1$, and $1\le j<a_2$, then the integers $k$ and  $j$  are called the level and the under-level of $x$, respectively.  Let $m$ be the least integer such  that $m\cdot a_3\in  Q $ with $m\ge 3$, $x_1:=a_3$ and let $x_i:=ia_3\notin Q$ for every $i$, $2\le i\le m-1$. By \rref{qq} we have  either   $i k<a_1$ (so $i j>a_2$) with $i j-(i-1) a_2<a_2 $,     which we refer to as an upgrade start  or $i j<a_2$ (so $i k>a_1$) with $i k-(i-1) a_1<a_1$  referred to as a downgrade start, where  $i$ verifies $2\le i\le s_1\le m-1$ for a certain $s_1$.
\end{definition}
Next we represent an algorithm to find $T(S)$.
\begin{definition}\label{kmm}Keeping the previous notation, if $h$ denotes the highest \textit{level} and $l$ the highest \textit{under-level} among $x_i$, then  our  pseudo-Frobenius numbers  candidates are of the form $F_0:=a_1 a_2-a_2-(l+1) a_1+w a_3$ for some $w\ge 0$ and $F_1:=a_1 a_2-(h+1) a_2-a_1+r a_3$ for some $r\ge 0$, where the integers $w$ and $r$ are the least integers such that $F_1+a_3\in Q$ and  $F_0+a_3\in Q$, respectively.

\end{definition}
It is worth mentioning that the two previous  pseudo-Frobenius candidates are the  precise formulas   for  pseudo-Frobenius numbers of $S$. This  is proved by Ralf Fr\"oberg et al.\ \cite[Theorem 11, second proof]{rou}  where  similar characterizations of pseudo-Frobenius numbers  for any   $S=\langle a_1,a_2,a_3\rangle$ are  given,  yielding  that  the \textit{type}  of  $\langle a_1,a_2,a_3\rangle$ is at most two.

\begin{example}\label{pppp}We illustrate the above definitions here; take $\langle 11,13,10\rangle$, $a_3=10=11\cdot 13-5\cdot 11-6\cdot 13$, and it can be verified that $m=5$, $(5\cdot a_3\in Q=\langle 11,13\rangle)$. The semigroup has a \textit{downgrade start} with  $s_1=2$, $h=7$, and $ l=10$. The pseudo-Frobenius numbers are $11\cdot13-11-8\cdot13+10=38$ and $11\cdot13-11\cdot11-13+ 20=29$.
\end{example}

The previous lemma  has the following extension:

\begin{theorem}\label{ww}Let $S=\langle a_1,\ldots,a_n\rangle$ and  $Q=\langle a_1,a_2\rangle$, where $a_1$ and $a_2$ are two coprime generators. If  $a_i+a_j\in Q$  for all $i$ and $j $ $\ge 3$, then the type of $S $ is at most $n-1$.
\end{theorem}

\begin{proof}The idea is simple in the sense that we get at most $n-1$ possible pseudo-Frobenius numbers   $F_i\notin S$  verifying both  $F_i+a_1 \in S$ and $F_i+a_2\in S$. By \lref{Se} we can write  $a_i:=a_1a_2-{\alpha}_ia_2-{\beta}_ia_1$ with $i\ge 3$ for some $1\le {\alpha}_i<a_1$ and $1\le {\beta}_i<a_2$. Arrange these generators so sequence $({\alpha}_i)_{i\ge 3}$ is decreasing (strictly) and thus, $({\beta}_i)_{i\ge 3}$ is  strictly increasing (since otherwise the difference between two consecutive minimal generators is in $Q$). Set  $$\begin{aligned}F_1&=a_1a_2-a_2-\underset{3\le i \le n}{\max}({\beta}_i+1)a_1,\\F_2&=a_1a_2-\underset{3\le i \le n}{\max}({\alpha}_i+1)a_2-a_1,\end{aligned}$$  and for $3\le i\le n-1$, set   $$F_{i}=a_1a_2-(\min({\alpha}_i,{\alpha}_{i+1})+1)a_2-(\min({\beta}_i,{\beta}_{i+1})+1)a_1.$$

The proof is direct for possible pseudo-Frobenius numbers having their \textit{level} or \textit{under-level} equals  $1$ (the numbers $F_1$ and  $F_2$ verify $F_1\notin S$, $F_2\notin S$, and $F_i+a_j\in S$ for any $i=1,2$ and $j=1,2)$.

Otherwise we have the result  from: $\{a_i+s| s\in Q, i\ge 3 \}\cup Q=S $ and  \rref{rk}. For example, one can prove that if $F_i=a_1 a_2-k a_2-j a_1$ is a pseudo-Frobenius number, $k>1$, and $j>1$, then  for some $i$ we have  ${\alpha}_{i+1}< k< {\alpha}_{i}$ and   ${\beta}_i< j< {\beta}_{i+1}$,  to  get the necessity of the given expressions.\end{proof}

\begin{remark}It is well known that  numerical semigroups $S$ of embedding dimension $n$ with  $t(S)\le n-1$ verify Wilf's conjecture; see \cite[Chapter 1]{w}.\end{remark}
\section{Algorithm's Applications}
Before  giving some examples and applications, we point out that we do not discuss or compare other algorithms with the one in \dref{kmm}. For example, the algorithm formulas \cite[Corollary 12]{san}, proved by giving \textit{Ap\'ery sets}, require additional parameters assuming pairwise coprime minimal generators.
\begin{theorem} \cite{j} \label{j}
Let $S=\langle a,b,c\rangle$, $a<b<c$, $d=\gcd(a,b)$, $n=\frac{bv}{d}-1$, $p=n(\frac{cdv}{a}-1)$, $v\in \mathbb{N}$ such that $bv\equiv d\pmod a$ with $vd<a$, then $$F(\langle a,b,c\rangle)=\dfrac{ab(F(\langle n,n+1,n+p\rangle)+2n+1)}{dn(n+1)}+(d-1)c-(a+b).$$

\end{theorem}
As stated by Igor Kan et al.\ \cite[Theorem 1]{j}, \tref{j} follows by applying Johnson's formula (\rref{w}) three times; it could be better seen if one starts  considering  $b=db'$, $a=da'$. Replacing the different variables $n,p$ by the given values, the computation should be applied two times to the semigroup $\langle n,n+1,n+p\rangle$. Lastly the formula is applied for $\langle a,b,c\rangle$ with $\gcd(a,b)=d$.
\begin{theorem}\cite{j}\label{jj} Let $S=\langle a,a+1,a+p\rangle$. If $a\ge p(p-4)+2$, then $$F(S)=\left\lfloor\frac{a}{p}\right\rfloor(a+1)+(p-1)\left\lfloor\frac{a+1}{p}\right\rfloor+(p-3)a-1.$$
\end{theorem}
\subsection{A downgrade semigroup}
Next we present particular semigroups that were introduced  by A. M. Robles-P\'erez and J. C. Rosales   \cite{fra}. 
\begin{lemma}\label{waw}Let $S=\langle a,a+1,ia+d\rangle$, $a=dk+j$,  $0< i< d<a$, and $1\le j<d$. If  $(k+1) i+j+1=d$, then the \textit{type} of $S$ is one. 
\end{lemma}
\begin{proof}The proof makes use of \rref{w}: by finding a common factor $f$ of $a+1$, $ia+d$ and showing that $a\in \langle \frac{a+1}{f}, \frac{ia+d}{f}\rangle$. Notice that $a=j(k+1)+k(i(1+k)+1)$ from the definition of $a$ and $d$ with $a+1=(d-i)(k+1)$ and $ia+d=(d-i)(1+i(1+k))$. The result follows taking $f=d-i$.

\end{proof}

\begin{theorem}\cite{fra}\label{l} Let $S=\langle a,b,c\rangle $ with pairwise coprime minimal generators and so $c=bh-ar=a b-(a-h) b-r a=(b-r)a-(a-h)b$. If $\dfrac{b}{r}\ge\left \lceil {\dfrac{a}{h}}\right \rceil$, then $T(S)=\left \{\left \lfloor\dfrac{a-1}{h}\right\rfloor c+ra-b-a, (a-1)b-\left \lfloor\dfrac{a-1}{h}\right\rfloor ra-a\right\}$. \end{theorem}

Using \rref{w} and \cite[Proposition 2.20]{w},  it is not hard to prove that  the formulas in the previous theorem are invariant with $\gcd(a,c)=d$ or $\gcd(b,c)=d$ and $d>1$ ($\gcd(a,b)=1$ necessarily). Consequently, by \tref{j}, we see that \tref{l} and \tref{sss} are equivalent for those numerical semigroups of \textit{type} 2. 
\begin{theorem}\label{sss}Let  $S=\langle a,a+1,ia+d\rangle$, $Q=\langle a,a+1\rangle$, where $a=dk+j$, $a_3=ia+d$,  $0< i< d<a$, $1\le j<d$, and $d\ge 3$.
If $(k+1)i+j\ge d-1$,  then  $T(S)=\{F_1,F_2\}$ for:  $$F_1=a(a+1)-(k(d-i)+1) a-(a+1)+w a_3,$$ $$F_2=a(a+1)-a-(a-d+1)(a+1)+(k-1)a_3,$$
 where $w=0$ if $(k+1)i+j\ge d$ and $w=k$ otherwise. \end{theorem}
\begin{proof}Write  $ia+d= a(a+1)-(d-i) a-(a-d)(a+1)$,  we are in a \textit{downgrade start} and $$na_3=a(a+1)-n(d-i) a-(a-dn)(a+1)$$ for all $n$  with $1\le n\le k$. The number    $(k+1) a_3\in Q$  since $$(k+1)a_3=a(a+1)-(a-(i(k+1)+j-d)) a+(d-j)(a+1).$$

Set by definition $wd:=3+x$, if $$F= a(a+1)-(k(d-i)+1) a-(a+1),$$ then  $F\notin S$ and $F+wa_3=(-2+j-x+wi-1+ki)(dk+j)+(2+x)(dk+j+1) $ is in $Q$ if, in particular: $$ki+j\ge w(d-i) \text{ and }2+x=wd-1\ge 0\quad (1),$$ $$\text{ or } -(a+1)<ki+j-w(d-i)<0 \text{  and }wd-1\ge a \quad (2).$$ We call  $w_0$ the smallest such integer  $\ge 1$;  but it can be verified 
that $w_0$ exists as $w_0=1$  if $(k+1)i+j\ge d$ from $(1)$ and $w_0= k+1$ otherwise $(2)$. Thus, first  $$F_1= a(a+1)-(k(d-i)+1) a-(a+1)+(w_0-1)(ia+d)\in T(S).$$
Take $G=a(a+1)-a-(a-d+1)(a+1)+k(ia+d)$ and set $y=d-1-j$, $G=a(ki+j)+y(a+1)\in Q$.
Now take $F=a(d-2)+d-1+(k-t)(ia+d)$ with $1\le t\le k$, clearly  $F-G=-(td+ita)$ thus, $$\begin{aligned}F&=a(ki+j+td+ita)-(-d+1+j+td+ita)(a+1)\\&=a(ki+j+td-it)-(1+j+(t-1)d)(a+1)\end{aligned}$$ and clearly $F$ is not  in $Q$ 
if and only if $ki+j+t(d-i)< a+1$  and $1+j+(t-1)d>0$. We call $t_0$ the smallest such integer $\ge 1$; but it can be verified that $(k-1)i+j+d<a+1$,  that is $(k-1)i<(k-1)d+1$ and so if $r=k-t_0=k-1,$ the second pseudo-Frobenius  is $F_2=a(d-2)+d-1+(k-1)(ia+d)$.

Rearranging terms we get $F_1=(j+ki-w(d-i)) a+(wd-1)(a+1) $  and $F_2=a((k-1)i+j+d)-(1+j)(a+1)$; if $w=0$,  then $F_1\neq F_2$ since $1<1+j$ and for $w=k$ we have $(k+1)i+j+1=d$ with $F_1=F_2$.
\end{proof}
\begin{remark}
By \rref{w}, the case $d$ divides $a$ in \tref{sss} is almost trivial.  Note also that the condition  $(k+1)i+j\ge d-1$ of \tref{sss} is equivalent to $$ a-(d-i)(k+1)\ge -1\Longleftrightarrow \dfrac{a+1}{d-i}\ge \bigg\lceil\dfrac{a}{d}\bigg\rceil \quad (\tref{l}).$$ See also  \cite[Corollary 1]{rec}.
\end{remark}

\begin{example}For $S=\langle 34,35,46\rangle$, we cannot apply \tref{jj},  applying \tref{sss} we have $d=12$, $a=34=2\cdot12+10$, $i=1$, $k=2$, and $j=10$. Here $w=0$  and $T(S)=\{373, 397\}$. Notice that  \tref{l} gives the same result without the  pairwise coprime condition.
\end{example}
\subsection{Semigroup of three consecutive squares}
The   Frobenius number formulas  of numerical semigroups generated by three consecutive squares were given by M. Lepilov et al.\ \cite{eee} and also by Fel \cite{ee}. Here we compute  their two pseudo-Frobenius numbers.
One can verify that $(n+1)(n+2)^2=(n+1)n^2+4(n+1)^2$  and $n(2n+1)^2=n(n+1)^2+(3n+2)n^2$ for any $n$. 
\begin{lemma}\label{svq}Let $S=\langle (4s-1)^2,(4s)^2,(4s+1)^2\rangle $ and $Q=\langle n^2,(n+1)^2\rangle$  with $n=4s-1$ and  $s>1$; we have $T
= \{256s^3-144s^2+8s-1,272s^3-168s^2+s-2\}$.
\end{lemma}
\begin{proof}First let us write $$(n+2)^2=n^2(n+1)^2- (n^2-2n-4)n^2-(4n-4)(n+1)^2;$$ 
if $n=4s-1$ and $s>1$, then we verify that $s(n+2)^2\in Q$ and $$(s-1)(n+2)^2=(16s^2-15s-1)n^2-(n^2-4n+3)(n+1)^2 \notin Q.$$  From  \dref{km} and \dref{kmm} notation, we have that $S$ has  an \textit{upgrade start} with $s_1=s-1$, $l=n^2-2n-4$, and $h=n^2-4n+3$.
Finally we have: $$\begin{aligned}&n^2(n+1)^2-(n^2-2n-3)n^2-(n+1)^2+(\alpha)(n+2)^2\\&=(4n+4+\alpha(4n+5))n^2-(1+\alpha(4n-4))(n+1)^2,\end{aligned}$$
 we verify
 $$n^2(n+1)^2-(n^2-2n-3)n^2-(n+1)^2+(s-1)(n+2)^2=(s-1)n^2+(16s-8)(n+1)^2$$ and $n^2(n+1)^2-n^2-(n^2-4n+4)(n+1)^2+(n+2)^2=16sn^2$.
\end{proof}

\begin{lemma}\label{szq}Let $S=\langle (4s-3)^2,(4s-2)^2,(4s-1)^2\rangle $ and $Q=\langle n^2,(n+1)^2\rangle$  with $n=4s-3$ and $s>1$; we have $T(S)= \{144s^3-296s^2+193s-42, 160s^3-352s^2+234s-51\}$. 
\end{lemma}
\begin{proof}
If $n=4s-3$ for some  $s>1$, then  $(2s-1)(n+2)^2\in Q$, the highest \textit{under-level} is for $$s(n+2)^2=n^2(n+1)^2-(16s^2-25s+8)n^2-(8s-9)(n+1)^2,$$ thus, we get $l=16s^2-25s+8$  and the highest  \textit{level} is for $h=16s^2-32s+16$ with $$(s-1)(n+2)^2=n^2(n+1)^2-(7s-3)n^2- (16s^2-32s+16)(n+1)^2$$  and $$(2s-2)(n+2)^2=n^2(n+1)^2-(14s-6) n^2-  (16s^2-40s+23) (n+1)^2.$$
Here $s_1=s-1$ and we verify that $$\begin{aligned}&n^2(n+1)^2-n^2- (16s^2-32s+17)(n+1)^2+  \alpha(n+2)^2\\&=(\alpha(4n+5)-1)n^2 -(8-8s+\alpha(4n-4))(n+1)^2,\end{aligned}$$ $$\begin{aligned}&n^2(n+1)^2-(16s^2-25s+9)n^2-(n+1)^2+\alpha(n+2)^2\\&=(9s-5+\alpha(4n+5))n^2 -(1+\alpha(4n-4))(n+1)^2,\end{aligned}$$

 $$n^2(n+1)^2-n^2- (16s^2-32s+17)(n+1)^2+  s(n+2)^2=(9(s-1)+4)n^2 +(n+1)^2,$$ and $ n^2(n+1)^2-(16s^2-25s+9)n^2-(n+1)^2+(s-1)(n+2)^2=(2(s-1))n^2   +(8(s-1))(n+1)^2$.
\end{proof}
\begin{lemma}\label{ssq}Let $S=\langle (4s)^2,(4s+1)^2,(4s+2)^2\rangle $ and $Q=\langle n^2,(n+1)^2\rangle$  with $n=4s$ and $s\ge2$; we have  $T(S)=\{112s^3+48s^2+8s-1,128s^3-20s-5\}$.
\end{lemma}
\begin{proof}Write  $(n+2)^2=n^2(n+1)^2- (n^2-2n-4)n^2-(4n-4)(n+1)^2$; here we need the following expressions: $$\left\{\begin{aligned}[l]s(n+2)^2&=n^2(n+1)^2-(3s+1)n^2-(16s^2-4s)(n+1)^2,\\ (s+1)(n+2)^2&=n^2(n+1)^2-(16s^2-5s-3)n^2-(12s-4)(n+1)^2,\\2s(n+2)^2&=n^2(n+1)^2-(2+6s)n^2-(16s^2-8s)(n+1)^2,\\(2s+1)(n+2)^2&=n^2(n+1)^2-(16s^2-2s-2)n^2-(8s-4)(n+1)^2, \\3s(n+2)^2&=n^2(n+1)^2-(9s+3)n^2-(16s^2-12s)(n+1)^2,\\(3s+1)(n+2)^2&=n^2(n+1)^2-(16s^2+s-1)n^2-(4s-4)(n+1)^2,\\n(n+2)^2&=n^2(n+1)^2-(12s+4)n^2-(16s^2-16s)(n+1)^2.\end{aligned}\right.$$ With $(n+1)(n+2)^2\in Q$, we have $s_1=s$, $h=16s^2-4s$, and $l=16s^2+s-1$.  Finally 
 $$\begin{aligned}&n^2(n+1)^2-n^2-(16s^2-4s+1)(n+1)^2+\alpha(n+2)^2\\&=(\alpha(4n+5)-1)n^2-(1-4s+16\alpha s-4\alpha)(n+1)^2,\end{aligned}$$  $$\begin{aligned}&n^2(n+1)^2-(16s^2+s)n^2-(n+1)^2+\alpha(n+2)^2\\&=  (2n+1-s+\alpha(4n+5))n^2 - (1+\alpha(4n-4))(n+1)^2,\end{aligned}$$

 $$n^2(n+1)^2-n^2-(16s^2-4s+1)(n+1)^2+(3s+1)(n+2)^2=3(n+1)^2+(7s+1)n^2,$$ and $$n^2(n+1)^2-(16s^2+s)n^2-(n+1)^2+s(n+2)^2=  (4s)n^2   +  (4s-1)(n+1)^2.$$

\end{proof}
\begin{lemma}\label{slq}Let $S=\langle (4s-2)^2,(4s-1)^2,(4s)^2\rangle $ and $Q=\langle n^2,(n+1)^2\rangle$  with $n=4s-2$ for $s>1$;  we have $T(S)= \{ 80s^3-80s^2+28s-5,128s^3-160s^2+60s-9\}$. 
\end{lemma}
\begin{proof}
We check that $(n+2)^2=n^2(n+1)^2- (n^2-2n-4)n^2-(4n-4)(n+1)^2$. The potential highest \textit{level} and \textit{under-level} are for the following multiples:$$\left\{\begin{aligned}(s-1)(n+2)^2&=n^2(n+1)^2-(11s-2)n^2-(16s^2-28s+12)(n+1)^2,\\ s(n+2)^2&=n^2(n+1)^2-(16s^2-13s+2)n^2-(4s-4)(n+1)^2,\\(2s-1)(n+2)^2&=n^2(n+1)^2-(6s-1)n^2-(16s^2-24s+8)(n+1)^2,\\(2s)(n+2)^2&=n^2(n+1)^2-(16s^2-18s+3)n^2-(8s-8)(n+1)^2, \\(3s-1)(n+2)^2&=n^2(n+1)^2-(s)n^2-(16s^2-20s+4)(n+1)^2,\\(3s)(n+2)^2&=n^2(n+1)^2-(16s^2-23s+4)n^2-(12s-12)(n+1)^2,\\n(n+2)^2&=n^2(n+1)^2-(12s-2)n^2-(16s^2-32s+12)(n+1)^2,\end{aligned}\right.$$ with $(n+1)(n+2)^2\in Q$. Consequently, $s_1=s-1$, $h=16s^2-20s+4$, and $l=16s^2-13s+2$.  Finally we have
$$\begin{aligned}&n^2(n+1)^2-n^2-(16s^2-20s+5)(n+1)^2+\alpha(n+2)^2\\&=(\alpha(4n+5)-1)n^2-(1-4s+16\alpha s-12\alpha)(n+1)^2,\end{aligned}$$  $$\begin{aligned}&n^2(n+1)^2-(16s^2-13s+3)n^2-(n+1)^2+\alpha(n+2)^2\\&=  (5s-2+\alpha(4n+5))n^2-(1+\alpha(4n-4))(n+1)^2.\end{aligned}$$

Thus,  $$n^2(n+1)^2-n^2-(16s^2-20s+5)(n+1)^2+s(n+2)^2=3(n+1)^2+(5s-2)n^2$$ and $$n^2(n+1)^2-(16s^2-13s+3)n^2-(n+1)^2+(3s-1)(n+2)^2=  (n)n^2   +  (n+1)(n+1)^2.$$
\end{proof}
\lref{sq} and \lref{sqa} can be considered the reduction by 4 ($d=4$ see \rref{w}) of \lref{ssq} and \lref{slq}, respectively .
\begin{lemma}\label{sq}Let $S=\langle (2s)^2,(4s+1)^2,(2s+1)^2\rangle $ and $Q=\langle n^2,(n+1)^2\rangle$  with $n=2s>2$; we have $T(S)=\{32s^3-12s^2-11s-2,28s^3-4s-1\}$. 
\end{lemma}
\begin{proof}First $$(2n+1)^2=n^2(n+1)^2-(2n+1)n^2-(n^2-2n-1)(n+1)^2.$$
 For $n=2s$, the semigroup $S$ has a \textit{downgrade start}  with $s_1=s-1$, $l=n^2-3s-1$, and $h=n^2-2n-1$. By computation we have $$\begin{aligned}&n^2(n+1)^2-n^2-(n^2-2n)(n+1)^2+\alpha(2n+1)^2\\&=(2n+\alpha(2n+1))(n+1)^2-(1+\alpha(2n+1))n^2,\end{aligned}$$

 $$n^2(n+1)^2-n^2-(n^2-2n)(n+1)^2+(s-1)(2n+1)^2=(s-1)(n+1)^2+(7s+1)n^2,$$ and $n^2(n+1)^2-(n^2-3s)n^2-(n+1)^2+ (2n+1)^2=3sn^2+4s(n+1)^2$.
\end{proof}

\begin{lemma}\label{sqa}Let $S=\langle (2s+1)^2,(4s+3)^2,(2s+2)^2\rangle $ and $Q=\langle n^2,(n+1)^2\rangle$  with $n=2s+1> 1$;   we have $T(S)=\{20s^3+28s^2+9s-1,32s^3+44s^2+13s-2 \}$. 
\end{lemma}
\begin{proof}First we got $$(2n+1)^2=n^2(n+1)^2-(2n+1)n^2-(n^2-2n-1)(n+1)^2.$$ Also if $n=2s+1$, then  $s(2n+1)^2\notin Q$ and $$(s+1)(2n+1)^2=(s+1)n^2+(3s+2)(n+1)^2,$$ so $s_1=s$, $l=n^2-s-1$, and $h=n^2-2n-1$. 
$$\begin{aligned}\text{From computation: } &n^2(n+1)^2-n^2-(n^2-2n)(n+1)^2+\alpha(2n+1)^2\\&=(2n+\alpha(2n+1))(n+1)^2-(1+\alpha(2n+1))n^2.\end{aligned}$$ 

By a direct verification
 $$n^2(n+1)^2-n^2-(n^2-2n)(n+1)^2+s(2n+1)^2=(5s+3)n^2+(3s+1)(n+1)^2$$ and $n^2(n+1)^2-(n^2-s)n^2-(n+1)^2+(2n+1)^2=sn^2+(4s+2)(n+1)^2$.
\end{proof} 
\subsection{A cyclic semigroup}
The next lemma is the particular three dimensional case  of Prof.\ Dao's question at: \url{http://mathoverflow.net/questions/292507/a-formula-for-frobenius-number-of-certain-numerical-semigroups}.
\begin{lemma}\cite[Proposition 10.16]{w} Let $a,b,c$ be three positive integers such that $ab+a+1$ and $ca+c+1$ are coprime. For $S= \langle ab+a+1,bc+b+1,ca+c+1\rangle$,   we have $T(S)=\{abc-1,2(abc-1)\}$.
\end{lemma}
\begin{proof}
First it is clear from the cyclic expression that, if any two of the generators  share a common factor $d$, the third one will be divisible by $d$. Assuming the hypothesis, we write
$$bc+b+1=(ab+a+1)(ca+c+1)-   c(ab+a+1)- ((b+1)a-b) (ca+c+1).$$   The generators are  $a_3=bc+b+1$ and  $Q=\langle ab+a+1,ca+c+1\rangle$.

We get $m=a+1$ (see \dref{km}) and $$\begin{aligned}(ab+a+1)(ca+c+1)-  & (ab+a+1)- ((b+1)a-b+1) (ca+c+1)+\alpha a_3\\&=(b+\alpha(b+1))(ca+c+1)-(1+\alpha c)(ab+a+1).\end{aligned}$$

Finally  we verify $$\begin{aligned}F_1&=(ab+a+1)(ca+c+1)-   (ab+a+1)- ((b+1)a-b+1) (ca+c+1)\\&+(a-1)(bc+b+1)=2(abc-1)\end{aligned}$$ and  $F_2=(ab+a+1)(ca+c+1)-   (ca+1)(ab+a+1)-  (ca+c+1)=abc-1$.
\end{proof}

\section{Acknowledgments}
 I  thank  the \url{http://www.les-mathematiques.net} site and the reviewers for the corrections and advice. Computations were assisted through Maple software.


\begin{thebibliography}{1}
\bibitem{fra}A.~M.~Robles-P\'erez and J.~C.~Rosales, The Frobenius problem for some numerical
semigroups with embedding dimension equal to
three, \textit{Hacettepe Journal of Mathematics and Statistics} \textbf{44} (2015), 901--908.
\bibitem{tm} A.~Tripathi, The Frobenius problem for modified arithmetic progressions, \textit{J.~Integer Seq.} \textbf{16(7)} (2013), Article 13.7.4.


\bibitem{j} I.~D.~Kan, B.~S.~Stechkin, and I.~V.~Sharkov, On the Frobenius problem for three arguments, \textit{Math. Notes} \textbf{62} (1997),  521--523.

\bibitem{san}J.~C.~Rosales and P.~A.~Garc\'ia-S\'anchez, Numerical semigroups with embedding dimension three, \textit{Archiv der Mathematik} \textbf{83} (2004),  488--496.

\bibitem{bran} J.~C.~Rosales and M.~B.~Branco, The Frobenius problem for numerical semigroups, \textit{Journal of Number Theory} \textbf{131} (2011),  2310--2319.

\bibitem{w} J.~C.~Rosales and  P.~A.~Garc\'ia-S\'anchez, \textit{Numerical Semigroups}, Developments in Mathematics \textbf{20}, Springer, New York, 2009.

\bibitem{fon} J.~L.~Ram\'irez Alfons\'in, \textit{The Diophantine Frobenius Problem}, Oxford Univ. Press, Oxford, 2005.
\bibitem{a} J.~J.~Sylvester. Excursus on rational fractions and partitions, \textit{Amer J. Math.} \textbf{5} (1882), 119--136.

\bibitem{fel} L.~G.~Fel, Frobenius Problem for Semigroups $S(d_1,d_2,d_3 )$,  \textit{Funct. Anal. Other Math.} \textbf{1}(2006), 119--157.


\bibitem{ee}L.~G.~Fel, Numerical semigroups generated by squares, cubes and quartics
of three consecutive integers, Conference paper, \textit{International meeting on numerical semigroups with applications},  Levico-Terme Italy, 2016.
\bibitem{eee}M.~Lepilov, J.~O'Rourke, and I.~Swanson, Frobenius numbers of numerical semigroups generated
by three consecutive squares or cubes, \textit{Semigroup Forum}, \textbf{91} (2015),  238--259.

 \bibitem{z} \"O.~J.~R\"odseth, On a linear Diophantine problem of Frobenius,  \textit{J. Reine Angew. Math.} \textbf{301}, (1978) 171--178.


\bibitem{rou} R.~Fr\"oberg, C.~Gottlieb, and  R.~H\"aggkvist, On numerical semigroups, \textit{Semigroup Forum} \textbf{35} (1987), 63--83.



\bibitem{joh} S.~M.~Johnson, A Linear Diophantine Problem, \textit{Can. J. Math.} \textbf{12}  (1960), 390--398.
\bibitem{rec}S.~S.~Batra, N.~Kumar, and A.~Tripathi, Some problems concerning the Frobenius number for
extensions of an arithmetic progression, \textit{The Ramanujan Journal}  \textbf{48} (2019), 545--565.


\end{thebibliography}
\end{document}